\documentclass[twoside,11pt]{elsarticle}
\usepackage{lineno,hyperref}
\usepackage{amssymb}
\usepackage{graphicx}
\usepackage{wrapfig}
\usepackage{calc}
\usepackage{amsmath,amsthm,amscd,amsfonts,amssymb,enumerate,relsize}
\modulolinenumbers[5]
\journal{Journal of \LaTeX\ Templates}
\usepackage{tikz}
\usepackage{pgf,tikz}
\usetikzlibrary{arrows}
\usepackage{calc}
\usepackage{tikz}
\usetikzlibrary{decorations.markings}
\tikzstyle{vertex}=[circle, draw, inner sep=0pt, minimum size=6pt]
\newcommand{\vertex}{\node[vertex]}
\newtheorem{theorem}{Theorem}[section]

\newtheorem{lemma}[theorem]{Lemma}

\begin{document}

\begin{frontmatter}

\title{Cayley graphs with metric dimension two - A characterization}
\author{Ebrahim Vatandoost\footnote{\it{Email: vatandoost@sci.ikiu.ac.ir}, Corresponding author}}
\author{Ali Behtoei\footnote{\it{Email: a.behtoei@sci.ikiu.ac.ir}}}
\author{Yasser Golkhandy Pour\footnote{\it {Email: y.golkhandypour@edu.ikiu.ac.ir}}}
\address{Department of Mathematical Sciences, Imam Khomeini International University, Qazvin, Iran}

\begin{abstract}
Let $\Gamma$ be a graph on $n$ vertices. A subset $W$ of $V(\Gamma)$ is called a resolving set when for each $u, v\in V(\Gamma)$ there exists $w\in W$ such that $\partial(u,w)\neq \partial(v, w)$. The metric dimension of $\Gamma$ is the minimum cardinality among resolving sets of $\Gamma$ and is denoted by $dim(\Gamma)$.
This  parameter has many applications in  chemistry, in  the  navigation  of  robots  in  networks  and 
in the problems  of  pattern  recognition  and  image 
processing  some of which involve the use of hierarchical data structures.
In this paper, we study the metric dimension of Cayley graphs. Specially, we present a complete characterization of Cayley graphs on Abelian groups whose metric dimension is two.
\end{abstract}

\begin{keyword}
\texttt{metric dimension}\sep \texttt{resolving set} \sep \texttt{Cayley graph}.
\MSC[2010] 05C12\sep 05C25 \sep 05C75.
\end{keyword}

\end{frontmatter}


\section{Introduction}
 Let $\Gamma$ be a (connected) graph with vertex set $V$ and edge set $E$. For each pair of vertices $u, v\in V$, the distance between $u$ and $v$ is the length of a shortest path between them and is denoted by $\partial(u,v)$. For convenient, we write $u\sim v$ when $\partial(u,v)=1$ and $u\nsim v$ otherwise. The neighbourhood of $u$ is  $N(u)=\{v\in V:~u\sim v\}$ and the diameter of $\Gamma$, denoted by $diam(\Gamma)$, is  $max\{\partial(u,v): u,v\in V\}$.
Given an ordered set $W=\{w_1, w_2, \ldots, w_k\}$ of vertices in $\Gamma$, the {\it metric representation} of  $u$ with respect to $W$ is the $k$-vector
  $r(u|W) = (\partial(u, w_1), \partial(u, w_2), \ldots, \partial(u, w_k))$.
If distinct vertices of $\Gamma$ have distinct metric representations, then $W$ is called a {\it resolving set} for $\Gamma$, see \cite{4} and \cite{5}. The {\it metric dimension} of $\Gamma$, $dim(\Gamma)$, is the minimum cardinality amoung resolving sets of $\Gamma$. 
The fractional metric dimension of permutation graphs is considered in \cite{Sinica-fractional}. In \cite{Sinica-GPetersen} the metric dimension of some family of generalized Petersen graphs are determined. Specially, it is shown that  each graph of the family of generalized Petersen graphs $P(n,4)$
has constant metric dimension.
For more results in this subject and related subjects see \cite{9}, \cite{Sinica-Strong} and \cite{6}.
One of the most interesting parts is to find some family of graphs that has constant metric dimension. In this regard, many families are studied and characterized. It is well known that paths has metric dimension one, cycles has metric dimension two and each complete graph on $n$ vertices has metric dimension $n-1$.
Caceres $et~al.$ in \cite{1} compute the metric dimension of fan $F_n$ and  the following theorem about the prism graph $P_m\times C_n$.
\begin{theorem}\cite{1}\label{11}
 For each $m$-vertex path  $P_m$ and $n$-vertex cycle $C_n$ we have
\begin{equation*}
dim(P_m\times C_n)=
      \begin{cases}
       2 &  n\equiv 1~(\!\!\!\!\!\!\mod 2),  \\
       3 &  n\equiv 0~ (\!\!\!\!\!\!\mod 2), m\geq 2.
      \end{cases}
\end{equation*}
\end{theorem}

Also, for the $M\ddot{o}bius~ Ladder$ graphs the following result is obtained. 
\begin{theorem}\cite{7}\label{21}
Let $M_n$ be a Mobius Ladder graph and $n\geq 8$ be an even integer. Then  we have $dim(M_n)=3$ for $n\equiv 2 ~(\!\!\!\!\mod 8)$ and $3\leq dim(M_n)\leq 4$ otherwise.
\end{theorem}

Let $G$ be a group and $S$ be a subset of $G$ which is closed under taking inverse and does not contain the identity element $e$. The {\it Cayley graph} $Cay(G,S)$ is a graph with vertex set $G$ and edge set
  $\{uv: vu^{-1}\in S\}$. Cayley graphs are regular and vertex transitive.
 In \cite{2} the metric dimension of a family of $3$-regular Cayley graphs is computed.
\begin{theorem}\label{12}\cite{2}
Let $G=\langle g\rangle$ be a cyclic group of order $n$ and $S=\{g, g^{-1}, g^{n/2}\}$. Then
\begin{equation*}\label{13}
dim(Cay(G,S))=
   \begin{cases}
     3 &  n\equiv 0 ~(\!\!\!\!\!\!\mod 4), \\
     4 &  n\equiv 2 ~(\!\!\!\!\!\!\mod 4).
   \end{cases}
\end{equation*}
\end{theorem}
In this paper, we provide a complete characterization of Cayley graphs on Abelian groups whose metric dimension is the constant value two. 
For this reason, we use two following useful results frequently.
 \begin{theorem}\cite{4}\label{02}
  Let $\Gamma$ be a graph and $\{u, v, w\}\subseteq V$ such that $u\sim v$ and $\partial(u,w)=d$. Then $\partial(v,w)\in\{d-1, d,d+1\}$.
\end{theorem}
\begin{theorem}\cite{10}\label{14}
Let $\Gamma$ be a graph with $dim(\Gamma)=2$ and $W=\{u, v\}$ be a resolving set for it. Then
\begin{itemize}
  \item[i)] there exists a unique shortest path between $u$ and $v$,
  \item[ii)] the degree of $u$ and the degree of $v$ are at most  three.
\end{itemize}
\end{theorem}
\section{\bf Preliminaries}
In this section, we provide some useful result's which will be applied  in the next section.
In \cite{103} the metric dimension of $k$-regular bipartite graphs for $k=n-1$ and $k=n-2$ is determined. 
We obtain a sharp lower bound for the metric dimension of $3$-regular bipartite graphs.
\begin{lemma}\label{22}
Let $\Gamma$ be a $3$-regular bipartite graph on $n$ vertices. Then $dim(\Gamma)\geq 3$.
\end{lemma}
\begin{proof}
Since $\Gamma$ is not a path, $dim(\Gamma)\geq 2$. Suppose on the contrary that $dim(\Gamma)=2$ and let $W=\{u, v\}$ be a resolving set for $\Gamma$. 
Also, assume that $\partial(u, v)=d$ and $N(u)=\{u_1, u_2, u_3\}$. By Theorem \ref{02}, $\partial(u_i, v)\in \{d-1, d, d+1\}$ for each
 $1\leq i\leq 3$. If there exist $i\neq j$ such that $\partial(u_i, v)=\partial(u_j, v)$, then $r(u_i|W)=r(u_j|W)$, which is a contradiction. 
 Without loss of generality, assume that $\partial(u_1, v)=d-1$, $\partial(u_2, v)=d$ and $\partial(u_3, v)=d+1$.
Let $\sigma_1$ be a shortest path between $u$ and $v$, and $\sigma_2$ be  a shortest path between $u_2$ and $v$.
Now two paths $\sigma_1$ and $\sigma_2$ using the edge $uu_2$ create an odd closed walk in $\Gamma$ which contains an odd cycle, a contradiction (see page~24 of \text{\cite{8}}). For sharpness, consider the hyper cube $Q_3$.
\end{proof}
\begin{theorem}\label{01} Let $G\ncong S_3$ be a group of order $n\geq 3$ and  $S\subset G$ be an inverse-closed generating subset of $G$ such that
 $e\notin S$ and $dim(Cay(G,S))=2$. Also, suppose that $Cay(G,S)$ is not a cycle and that $W$ is an optimal resolving set for $Cay(G,S)$. Then we have $W\cap S=\emptyset$.
\end{theorem}
\begin{proof}
Since $Cay(G,S)$ is vertex transitive, without loss of generality, we can assumed that $e\in W$ and  $W=\{e, w\}$ for some $w\in G$. Since $dim(Cay(G,S))=2$, Theorem \ref{14} implies that $|S|\leq 3$. Since $n>2$, we have $|S|\neq 1$.
If $|S|=2$, then $Cay(G,S)$ is isomorphic to a cycle which contradicts the assumptions. Hence $|S|=3$.
On the contrary, assume that $W\cap S\neq \emptyset$ and $S=\{u, v, w\}$.
According to the order of $w$, the proof falls into the following two cases.
\newline
{\bf Case 1.} $O(w)=2$: In this case, we first claim that $N(w)\cap S=\emptyset$. Note that $N(w)=\{e, uw, vw\}$. If $uw\in S$, then $uw=u$, $uw=v$, or $uw=w$. Since $u\neq e$ and $w\neq e$, it follows that $uw=v$. Thus, in $Cay(G,S)$, $u\sim w$ and $v\sim w$. Hence $r(u\mid W)=r(v\mid W)=(1, 1)$,
which is a contradiction.
If $vw\in S$, then similarly we can obtain a contradiction. This completes the proof of claim. Hence $\partial(e, uw)=\partial(e, vw)=2$ and so $r(uw\mid W)=r(vw\mid W)=(2, 1)$, a contradiction.
\newline
{\bf Case 2.}  $O(w)\neq 2$: Since $S=S^{-1}$ and $w^{-1}\in S$, without loss of generality, we can assume that $w^{-1}=v$ and $O(u)=2$. We claim that $N(w)\cap S=\emptyset$. Note that $N(w)=\{e, uw, w^2\}$.
At first, assume that $uw\in S$. Then $uw=u$, $uw=w$, or $uw=w^{-1}$. Since $u\neq e$ and $w\neq e$, $uw=w^{-1}$. Hence $w^2=u$. Since $O(u)=2$, $O(w)=4$. In addition, since $G=\langle S\rangle$, $G$ is isomorphic to $\mathbb{Z}_4$ (cyclic group of order four) and hence, $Cay(G,S)$ is isomorphic to $K_4$. Therefore, $dim(Cay(G,S))=3$ which is a contradiction.
Now assume that $w^2\in S$. Then $w^2=u$, $w^2=w$, or $w^2=w^{-1}$. It is clear that $w^2\neq w$. If $w^2=u$, then $O(w)=4$, and so $G$ is isomorphic to $\mathbb{Z}_4$, and so $Cay(G,S)$ is isomorphic to $K_4$. Hence $dim(Cay(G,S))=3$, which produce a contradiction.
\newline
If $w^2=w^{-1}$, then $O(w)=3$. First suppose $G$ is an Abelian group. Then $G$ is isomorphic to $\mathbb{Z}_6$, cyclic group of order six.
By Theorem \ref{12}, $dim(Cay(\mathbb{Z}_6,S))=4$, which is a contradiction.
\newline
Next assume that $G$ is a non-Abelian group.
In this case all possible neighbors of vertices $u, w, w^2$ and $uw$ are depicted in Fig. \ref{fig2}. Define $L=\{wuw, w^2uw\}$. It is claimed that $L\cap S=\emptyset$.
 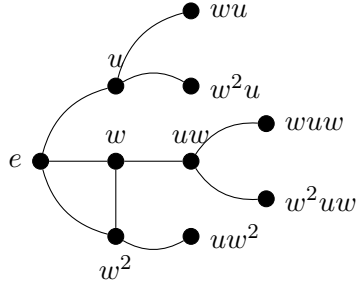
\begin{figure}[htb]
\begin{center}
\[\begin{tikzpicture}
 \vertex (a) at (1,0) [label=left:$e$][fill=black] {};
 \vertex (b) at (2,0) [label=above:$w$] [fill=black]{};
 \vertex (c) at (3,0) [label=above:$uw$] [fill=black]{};
  \vertex (d) at (4,.5) [label=right:$wuw$] [fill=black]{};
 \vertex (e) at (4, -.5) [label=right:$w^2uw$] [fill=black]{};
 \vertex (f) at (2, 1) [label=above:$u$] [fill=black]{};
 \vertex (g) at (3,2) [label=right:$wu$] [fill=black]{};
 \vertex (h) at (3, 1) [label=right:$w^2u$] [fill=black]{};
 \vertex (m) at (2, -1) [label=below:$w^2$] [fill=black]{};
 \vertex (n) at (3, -1) [label=right:$uw^2$] [fill=black]{};
 \path
(a) edge[bend left] (f)
(a) edge (b)
(a) edge[bend right] (m)
(f) edge[bend left] (g)
(f) edge[bend left] (h)
(b) edge (c)
(b) edge (m)
(m) edge[bend right] (n)
(c) edge[bend left] (d)
(c) edge[bend right] (e);
\end{tikzpicture}\]
\caption{{\footnotesize A part of $Cay(G,S)$ in that the neighborhood of letters in $S$ is depicted.}}
\label{fig2}
\end{center}
\end{figure}
\newline
It is easy to check that $wuw\neq w$ and $wuw\neq w^2$. Also if $wuw=u$, then $O(uw)=2$. Thus
$G=\langle S\rangle=\langle u,w : u^2=w^3=(wu)^2=e\rangle$, and so $G\cong S_3$, which is wrong.
\newline
On the other hand, it is clear that $w^2uw\neq w$ and $w^2uw\neq w^2$. Also if $w^2uw=u$, then $uw=wu$ and so $O(uw)=6$. Thus $G\cong \mathbb{Z}_6$ which produce a contradiction. Therefore $L\cap S=\emptyset$, as claimed.
 \newline
Now Assume that $K=\{wu, w^2u, uw^2\}$. It will be shown that $L\cap K=\emptyset$.
\newline
First let $wuw\in K$. Since $w\neq e$, $wuw\neq wu$. Hence $wuw=uw^2$ or $wuw=w^2u$. In any way, $uw=wu$ and so $G\cong \mathbb{Z}_6$ which produce a contradiction.
\newline
Next suppose that $w^2uw\in K$. Since $w\neq e$, $w^2uw\neq w^2u$. Let $w^2uw=wu$ or $w^2uw=uw^2$. Then $O(uw)=2$; and since
$G=\langle S\rangle$, $G=\langle u,w : u^2=w^3=(uw)^2=e\rangle$. Hence $G\cong S_3$, which is a contradiction.
 \newline
  Thus $L\cap S=L\cap K=\emptyset$. So $\partial(e, wuw)=\partial(e, w^2uw)=3$ and $\partial(w, wuw)=\partial(w, w^2uw)=2$. We have
  $r(wuw\mid W)=r(w^2uw\mid W)=(3,2)$, which contradicts the fact that $W$ is a resolving set for $Cay(G,S)$.
\newline
In all cases for $w$, a contradiction was produced; and it follows that $W\cap S=\emptyset$ and completed the proof.
\end{proof}
\section{\bf Cayley graphs with metric dimension two}
In this section, we present a complete characterization for Cayley graphs on Abelian groups whose metric dimension is two.
\begin{theorem}\label{25}
Let $G=\langle u\rangle$ be a cyclic group of order $n$ and $S=\{u^i, u^{-i}, u^{n/2}\}$ be a generating subset for $G$. Then $dim(Cay(G,S))=2$ if and only if $gcd(i, n/2)=1$ and $n\equiv 2(\mod 4)$.
\end{theorem}
\begin{proof}
At first, assumed that $dim(Cay(G,S))=2$. Since $S$ is a generating subset for $G$, we have $gcd(i,n)=1$ or $gcd(i,n/2)=1$. If $gcd(i,n/2)\neq 1$, then
$gcd(i,n)=1$. Hence $O(u^i)=n$ and $G=\langle u^i\rangle$. Now Theorem \ref{12} implies that $dim(Cay(G,S))\in\{3,4\}$ which is a contradiction.
Therefore $gcd(i,n/2)=1$ and so $gcd(i,n)=2$. This means that $i$ is even and $n/2$ is odd denoted it by $2k+1$. Hence $n\equiv 2 (\mod 4)$.
 \newline
Next, let $gcd(i,n/2)=1$ and $n\equiv 2(\mod 4)$. Then $O(u^i)=n/2$ which is an odd. If $H=\langle u^i\rangle$, then since $n/2$ is an odd and $O(u^{n/2})=2$, $u^{n/2}\notin H$. Hence $[G:H]=2$ and so $G=H\cup Hu^{n/2}$.
\newline
Obviously, $Cay(G,S)$ contains two disjoint cycles on $n/2$ vertices
 $e\sim u^i\sim u^{2i}\sim \ldots\sim (u^{i})^{n/2-1}\sim e$  and $u^{n/2}\sim u^{i+n/2}\sim u^{2i+n/2}\sim \ldots\sim u^{(n/2-1)i+n/2}\sim u^{n/2}$ as a subgraph.
On the other hand, for each $1\leq k\leq n/2$, $u^{ki}\sim u^{ki+n/2}$. Since $|S|=3$, $Cay(G,S)$ is isomorphic to the prism graph, $P_2\times C_{2k+1}$. See Fig. \ref{fig5}, for more details. Thus by Theorem \ref{11}, $dim(Cay(G,S))=2$.
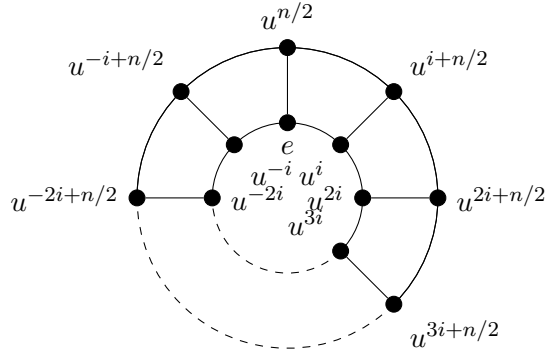
\begin{figure}[htb]
\begin{center}
\[\begin{tikzpicture}
 \vertex (a) at (90:2) [label=above:$\footnotesize{u^{n/2}}$][fill=black] {};
 \vertex (b) at (45:2) [label=above right:$\footnotesize{u^{i+n/2}}$] [fill=black]{};
 \vertex (c) at (360:2) [label=right:$\footnotesize{u^{2i+n/2}}$] [fill=black]{};
  \vertex (d) at (-45:2) [label=below right:$\footnotesize{u^{3i+n/2}}$] [fill=black]{};
 \vertex (f) at (180:2) [label=left:$\footnotesize{u^{-2i+n/2}}$] [fill=black]{};
\vertex (g) at (135:2) [label= above left:$\footnotesize{u^{-i+n/2}}$] [fill=black]{};
 \vertex (aa) at (90:1) [label=below:$\footnotesize{e}$][fill=black] {};
 \vertex (bb) at (45:1) [label=below left:$\footnotesize{u^i}$] [fill=black]{};
 \vertex (cc) at (360:1) [label=left:$\footnotesize{u^{2i}}$] [fill=black]{};
  \vertex (dd) at (-45:1) [label=above left:$\footnotesize{u^{3i}}$] [fill=black]{};
 \vertex (ff) at (180:1) [label=right:$\footnotesize{u^{-2i}}$] [fill=black]{};
\vertex (gg) at (135:1) [label=below right:$\footnotesize{u^{-i}}$] [fill=black]{};
 \path
(a) edge (aa)
(b) edge (bb)
(c) edge (cc)
(d) edge (dd)
(f) edge (ff)
(g) edge (gg);
\draw [dashed] (-45:2) arc (-45:-180:2);
\draw  (90:2) arc (90:45:2);
\draw (45:2) arc (45:0:2);
\draw  (0:2) arc (0:-45:2);
\draw (90:2) arc (90:135:2);
\draw  (135:2) arc (135:180:2);
\draw  (90:1) arc (90:45:1);
\draw  (45:1) arc (45:0:1);
\draw  (0:1) arc (0:-45:1);
\draw  (90:1) arc (90:135:1);
\draw  (135:1) arc (135:180:1);
 \draw
[dashed] (-45:1) arc (-45:-180:1);
\draw  (180:2) arc (180:-45:2);
\end{tikzpicture}\]
\caption{{\footnotesize $Cay(G,S)$ is isomorphic to the prism graph, $P_2\times C_{2k+1}$.}}
\label{fig5}
\end{center}
\end{figure}
\end{proof}
\begin{theorem}\label{24} Let $G$ be a non-cyclic Abelian group of order $n>4$; and let $S$ be a generating subset of $G$ in which $e\notin S=S^{-1}$. Then $dim(Cay(G,S))\neq 2$.
\end{theorem}
\begin{proof} Suppose on the contrary that $dim(Cay(G,S))=2$. By Theorem \ref{14}-(ii), $|S|\leq 3$. $G$ is not cyclic and hence $|S|>1$. 
If $|S|=2$ and $S=\{u, v\}$, then $S=S^{-1}$ implies that $O(u)=O(v)=2$ or $u=v^{-1}$. Since $S$ is a generating set for $G$ and $G$ is a non-cyclic group, $u\neq v^{-1}$. Thus, since $G=\langle S\rangle$ is non-cyclic, $G$ is isomorphic to $\mathbb{Z}_2\times\mathbb{Z}_2$ which contradicts the fact that $n>4$.
\newline
 Now we can assume that $|S|=3$. By vertex transitivity of Cayley graphs, let $W=\{e, w\}$ be a resolving set for $Cay(G,S)$. By Theorem \ref{01}, $W\cap S=\emptyset$; and we can assume that $S=\{u, v, z\}$.
\newline
First, let $O(u)=O(v)=O(z)=2$. Since $G$ is an Abelian group, $n\in\{4, 8\}$ and this using the condition $n>4$, implies that $n=8$. 
In this case, $G$ is isomorphic to
$\mathbb{Z}_2\times\mathbb{Z}_4$ or $\mathbb{Z}_2\times\mathbb{Z}_2\times\mathbb{Z}_2$. In each case, $Cay(G,S)$ is isomorphic to $P_2\times C_4$. See Fig. \ref{fig8} for more details.
\begin{figure}[htb]
\begin{center}
\[\begin{tikzpicture}
 \vertex (a) at (0,2) [label=above:$e$][fill=black] {};
 \vertex (b) at (0,-2) [label=below:\footnotesize{$uv$}] [fill=black]{};
 \vertex (c) at (2,0) [label=right:$u$] [fill=black]{};
  \vertex (d) at (-2,0) [label=left:$v$] [fill=black]{};
 \vertex (aa) at (0,1) [label=below:$z$][fill=black] {};
 \vertex (bb) at (0,-1) [label=above:\footnotesize{$uvz$}] [fill=black]{};
 \vertex (cc) at (1,0) [label=left:\footnotesize{$uz$}] [fill=black]{};
  \vertex (dd) at (-1,0) [label=right:\footnotesize{$zv$}] [fill=black]{};
  \path
(a) edge (aa)
(b) edge (bb)
(c) edge (cc)
(d) edge (dd);
\draw (0,0) circle (1) (0,0) circle (2);
\end{tikzpicture}\]
\caption{{\footnotesize It was assumed $G$ is an Abelian group of order $8$ and $S=\{u, v, z\}$ such that $O(u)=O(v)=O(z)=2$.}}
\label{fig8}
\end{center}
\end{figure}
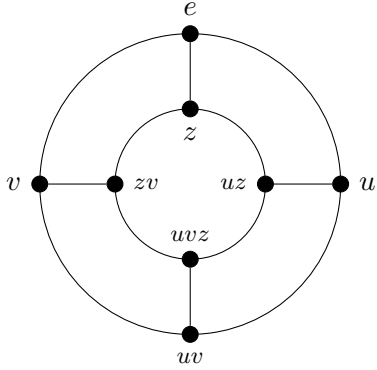
 Hence by Theorem \ref{11}, we have $dim(Cay(G,S))=3$, which is a contradiction. Thus $S=\{u, u^{-1}, v\}$ such that $z=u^{-1}$, $O(u)=t\geq 2$ and $O(v)=2$.
 Since $S$ is a generating subset of $G$ and $w\notin S$, we have $w=u^kv$ for some
  $1\leq k\leq t-1$; or $w=u^{k}$ for some $2\leq k\leq t-2$. The following cases will be considered.
 \newline
 {\bf Case 1.} $w=u^kv$ for some $1\leq k\leq t-1$.
  \newline
  Since $u^{-\ell}=u^{n-\ell}$ for any integer $\ell$, by renaming $u=u^{-1}$ if it is necessary, it can be assumed $k\leq n/2$ is a positive integer. Certainly, $e\sim u\sim u^2\sim\ldots\sim u^k\sim u^kv$ is a path of length $k+1$ from $e$ to $u^kv$, and so $\partial(e, u^kv)\leq k+1$. If $\partial(e, u^kv)\leq k$, then (by definition of Cayley graphs) there exist a shorter path whose vertices are created by linear combination of powers of $u$ and $v$, or of $u^{-1}$ and $v$.
 \newline
 First, let the shorter path be created by combination of powers of $u$ and $v$. Since $G$ is Abelian there is a positive integer $\ell<k$ such that $u^kv=u^{\ell}v$. Hence $u^k=u^{\ell}$ which is not possible.
 \newline
 Next, suppose the shorter path be created by combination of powers of $u^{-1}$ and $v$. Then there is a positive integer
 $\ell<k$ such that $u^kv=u^{-\ell}v$. Hence $u^k=u^{-\ell}$. Since $u^{-\ell}=u^{n-\ell}$, we have $k=n-\ell\gneqq n-n/2$, which is a contradiction.
 \newline
 Therefore $\partial(e, u^kv)=k+1$. Now, we obtain two distinct path of length $k+1$ from e to $u^kv$ as depicted in Fig. \ref{fig6} which contradicts Theorem \ref{14}.
\begin{figure}[htb]
\begin{center}
\[\begin{tikzpicture}
 \vertex (1) at (-3,0) [label=left:$e$] [fill=black]{};
 \vertex (2) at (-2,.75) [label=above:$u$] [fill=black]{};
 \vertex (3) at (-1,.95) [label=above:$u^2$] [fill=black]{};
 \vertex (4) at (2,.75)  [label=above:$u^k$][fill=black]{};
 \vertex (5) at (2,-.75)  [label=below:$u^{k-1}v$][fill=black]{};
 \vertex (6) at (-1,-.95) [label=below:$uv$] [fill=black]{};
 \vertex (7) at (-2,-.75) [label=below:$v$] [fill=black]{};
 \vertex (8) at (3,0) [label=right:$u^kv$] [fill=black]{};
 \path
 (3)edge[bend left=10, dashed, black, ultra thick] (4)
  (6)edge[bend right=10, dashed, black, ultra thick] (5)
;
\draw[-latex] (0, 0) ellipse [x radius = 3cm, y radius = 1cm,
  start angle = 30, end angle = 150];

\end{tikzpicture}\]
\caption{{\footnotesize Two distinct paths of length $k+1$ between $e$ and $u^kv$.}}
\label{fig6}
\end{center}
\end{figure}
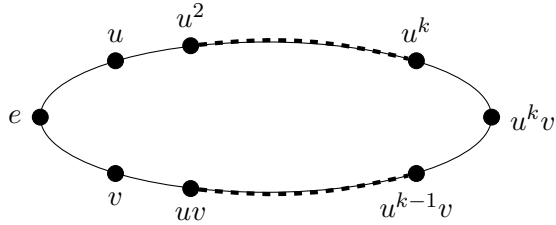
\newline
 {\bf Case 2.} $w=u^k$ for some $2\leq k\leq t-2$.
  \newline
  Since $u^{-\ell}=u^{t-\ell}$ for any integer $\ell$, by renaming $u=u^{-1}$ if it is necessary, it can be assumed $k\leq t/2$ is a positive integer. It is obvious that
 $e\sim u\sim u^2\sim \ldots\sim u^{k-1}\sim u^{k}$ is a path of length $k$ from $e$ to $u^k$,
  which is denoted by $P$. Hence $\partial(e,u^k)\leq k$.
 \newline
 If $\partial(e,u^k)<k$, then there exists a path from $e$ to $u^k$ with length shorter than $k$ which must be created by combination of $u$ and $v$. Assume the new path is separated from $P$ in $i$-th, and joint to $P$ in $j$-th vertex again, for some $0\leq i<j\leq k$. See Fig. \ref{fig7}, for more details.
 \begin{figure}[htb]
\begin{center}
\[\begin{tikzpicture}
 \vertex (1) at (0,0) [label=left:$e$] [fill=black]{};
 \vertex (2) at (1,0) [label=above:$u$] [fill=black]{};
 \vertex (5) at (2,0) [label=above:$u^i$] [fill=black]{};
 \vertex (6) at (3,0)  [fill=black]{};
 \vertex (7) at (4,0)  [fill=black]{};
 \vertex (8) at (5,0) [label=above:$u^j$] [fill=black]{};
 \vertex (11) at (6,0) [label=right:$u^k$] [fill=black]{};
 \vertex (12) at (2.5,-.5) [label=below:$u^iv$] [fill=black]{};
 \vertex (15) at (4.5,-.5) [label=below:$u^jv$] [fill=black]{};
 \path
 (1) edge (2)
 (2) edge[dashed] (5)
 (5) edge (6)
 (7) edge[dashed] (6)
 (7) edge (8)
 (8) edge[dashed] (11)
 (15) edge[bend right=20] (8)
 (5) edge[bend right=20] (12)
 (15) edge[bend left=20, dashed] (12);
\end{tikzpicture}\]
\caption{{\footnotesize }}
\label{fig7}
\end{center}
\end{figure}\\
\newline
Clearly, the new path has two vertices $u^iv$ and $u^jv$ more than $P$, which produce a contradiction. Therefore $\partial(e,u^k)=k$. Consequently, it can be assumed that $k\neq n/2$. If not, then we can obtain two distinct path of length $n/2$ from $e$ to $u^k$, which contradicts Theorem \ref{24}-(i).
\newline
By using the structure of $P$, we have $N(u^k)=\{u^{k-1}, u^{k+1}, u^kv\}$ and $\partial(e, u^{k-1})=k-1$.
In addition, If $k+1>t/2$, then $k\geq t/2$ which contradicts the fact $k<t/2$. Hence $k+1\leq t/2$ and so $\partial(e, u^{k+1})=k+1$.
\newline
Finally, By Theorem \ref{02}, $\partial(e, u^kv)=k$. On the other hand, by similar proof of Case $1$, we have $\partial(e, u^kv)=k+1$, which produce a contradiction. Therefore $W$ can not be a resolving set for $Cay(G,S)$ and so $dim(Cay(G,S))\neq 2$. 
\end{proof}
Now, we are ready to establish our main Theorem.
\begin{theorem}\label{99}
  Let $G$ be an Abelian group of order $n>4$ and $S\subset G$, where $e\notin S=S^{-1}$. Then $dim(Cay(G,S))=2$ if and only if $G=\langle u\rangle$ is cyclic and $S=\{u^i, u^{-i}, u^{n/2}\}$ in which $gcd(i,n/2)=1$ and $n\equiv 2(\mod 4)$.
\end{theorem}
\begin{proof}
if $G=\langle u\rangle$; and $S=\{u^i, u^{-i}, u^{n/2}\}$ in which $gcd(i,n/2)=1$ and $n\equiv 2(\mod 4)$, then by Theorem \ref{25}, $dim(Cay(G,S))=2$.
\newline
On the other hand, suppose that $dim(Cay(G,S))=2$. If $G$ is a non-cyclic Abelian group, then by Theorem \ref{24}, $dim(Cay(G,S))\neq 2$ which is wrong. Hence $G$ is a cyclic group; and the results will be obtained from Theorem \ref{25}.
\end{proof}
\section*{References}

\bibliography{mybibfile}
\end{document}